\documentclass[review,authoryear]{elsarticle}

\makeatletter
\def\ps@pprintTitle{%
 \let\@oddhead\@empty
 \let\@evenhead\@empty
 \def\@oddfoot{\hfill\today}%
 \let\@evenfoot\@oddfoot}
\makeatother

\usepackage{lineno,hyperref}

\usepackage{amsmath,amsthm,amssymb,bm}
\usepackage{graphicx}
\usepackage{tikz}
\usepackage{tikz-qtree}
\usepackage{cleveref}
\usepackage{algpseudocode,algorithm}
\algnewcommand\algorithmicinput{\textbf{Input:}}
\algnewcommand\Input{\item[\algorithmicinput]}
\algnewcommand\algorithmicoutput{\textbf{Output:}}
\algnewcommand\Output{\item[\algorithmicoutput]}
\usepackage{caption}

\newtheorem{thm}{Theorem}
\newtheorem{lem}[thm]{Lemma}
\newtheorem{prop}[thm]{Proposition}

\theoremstyle{definition}
\newtheorem{defn}{Definition}
\newtheorem{expl}{Example}
\newtheorem{rem}{Remark}

\newcommand{\K}{\mathbb{K}}
\newcommand{\R}{\mathbb{R}}
\newcommand{\ann}{\textrm{Ann}}
\newcommand{\fl}{\textrm{FL}}

\newcommand{\mvh}{\textsc{Matrix-vector-horner}}
\newcommand{\vmh}{\textsc{Vector-matrix-horner}}

\begin{document}

\begin{frontmatter}
  \title{Fast Algorithm for Calculating the Minimal Annihilating
    Polynomials of Matrices via Pseudo Annihilating Polynomials}
  \author[tsukuba]{Shinichi Tajima}
  \ead{tajima@math.tsukuba.ac.jp}

  \author[kanazawa]{Katsuyoshi Ohara}
  \ead{ohara@se.kanazawa-u.ac.jp}
  \ead[url]{http://air.s.kanazawa-u.ac.jp/~ohara/}

  \author[tsukuba]{Akira Terui\corref{corauthor}}
  \cortext[corauthor]{Corresponding author}
  \ead{terui@math.tsukuba.ac.jp}
  \ead[url]{http://researchmap.jp/aterui}

  \address[tsukuba]{Faculty of Pure and Applied Sciences, University
    of Tsukuba, Tsukuba 305-8571, Japan}
  \address[kanazawa]{Faculty of Mathematics and Physics, Kanazawa
    University, Kanazawa 920-1192, Japan}

  \begin{abstract}
    Minimal annihilating polynomials are very useful in a wide variety
    of algorithms in exact linear algebra.  A new efficient method is
    proposed for calculating the minimal annihilating polynomials for
    all the unit vectors, for a square matrix over a field of
    characteristic zero. Key ideas of the proposed method are the
    concept of pseudo annihilating polynomial and the use of binary
    splitting technique. Efficiency of the resulting algorithms is
    shown by arithmetic time complexity analysis.
  \end{abstract}

\end{frontmatter}

\newpage

\section{Introduction}
\label{sec:intro}

Exact linear algebra calculations are of importance in various fields
in mathematics, and a variety of softwares has been developed 
(\cite{M4RI}, \cite{linbox2002}, \cite{dum-gio-per2008},
\cite{har2010}, \cite{linbox152}, \cite{par2017}, \cite{ntl11},
\cite{iml105}).

We have proposed, in the context of exact linear algebra, a series of
algorithms on eigenproblems such as calculating
eigenvectors (\cite{taj-oha-ter2015}, \cite{taj-ter2013},
\cite{taj-ter2014}, \cite{taj-ter2015}, \cite{ter-taj2011}),
generalized eigenvectors and/or eigen space
decomposition (\cite{oha-taj2009b}, \cite{oha-taj2012},
\cite{oha-taj2013}, \cite{oha-taj2014}, \cite{taj2013},
\cite{taj-hig2009}, \cite{taj-iiz2009}), matrix inverse
(\cite{taj-oha-ter2016b}, \cite{taj-oha-ter2016}), spectral
decomposition (\cite{taj-iiz2012}, \cite{oha-taj2009b},
\cite{oha-taj2009}, \cite{oha-taj2012b}, \cite{oha-taj2012},
\cite{taj2012}), and so on. We have shown in these papers that the use
of the minimal annihilating polynomials of all unit vectors for a matrix
provides effective methods for these eigenproblems. We have shown in
particular that the cost of computation can be reduced significantly
by using the minimal annihilating polynomials of a matrix for all unit
vectors, in place of the minimal polynomial. It
is therefore desirable, in the context of symbolic computation, to
construct efficient methods for calculating all the minimal
annihilating polynomials for designing efficient algorithms on
eigenproblems.



We propose in this paper an efficient method for calculating all the
unit minimal annihilating polynomials for a matrix.
Keys of our approach are the notion of pseudo annihilating introduced
in Section~\ref{sec:pmap} and the use of binary splitting technique. More
precisely, in the proposed method, we first calculate, almost
deterministic way, pseudo annihilating polynomials as candidates of
unit minimal annihilating polynomials, nextly, we certify that these are
true unit minimal annihilating polynomials. Note that since the method
for calculating pseudo annihilating polynomials can be regarded as a
kind of extension or variant of Wiedemann's method (\cite{wie1986})
for computing
characteristic polynomials, the possibility that the pseudo
annihilating polynomials are true unit minimal annihilating polynomials is
expected to be high. This is one of advantage of our method. Note also
that, for the case where a pseudo annihilating polynomial does not
pass verification, one can efficiently revise it to obtain true one,
because each pseudo annihilating polynomial enjoys the property that
it divides the corresponding minimal annihilating polynomial.  Another
advantage of the 
proposed method lies in the fact that certain processes in the main
blocks of the algorithm can be executed in parallel.  Therefore, the
resulting algorithm fits computing environments of multiple processors
and/or cores to gain its efficiency.

The rest of the paper is organized as follows.  In
Section~\ref{sec:tmap}, we briefly recall the notion of minimal
annihilating polynomial and fix notation. We give a naive algorithm
for calculating minimal annihilating polynomials.  In
Section~\ref{sec:pmap}, we introduce the notion of pseudo annihilating
polynomials and present an algorithm for calculating pseudo
annihilating polynomials of all unit vectors.  We also give an
algorithm for calculating, via pseudo annihilating polynomials, all the
unit minimal annihilating polynomials for a matrix. In
Section~\ref{sec:upap}, by adopting binary splitting technique, we describe an
efficient method for calculating pseudo unit annihilating
polynomials. We also show, by
time complexity analysis, that the use of binary splitting technique
improves the efficiency of calculation.


\section{The minimal annihilating polynomials}
\label{sec:tmap}

Let $A$ be a $n\times n$ matrix over a field $\K$ of characteristic
zero and let $\chi_A(\lambda)$ be the characteristic polynomial of
$A$. For a non-zero column vector $\bm{v}$ over $\K$ of dimension $n$,
let $\ann_{\K[\lambda]}(A,\bm{v})\subset \K[\lambda]$ denote the ideal
defined to be
\begin{equation}
  \label{eq:annideal}
  \ann_{\K[\lambda]}(A,\bm{v}) = \{p(\lambda)\in \K[\lambda]\mid
  p(A)\bm{v}=\bm{0}\},
\end{equation}
where $\K[\lambda]$ is the ring of univariate polynomial in $\lambda$
over $\K$.  The monic generator, denoted by $\pi_{A,\bm{v}}(\lambda)$,
of the ideal $\ann_{\K[\lambda]}(A,\bm{v})$ is called \emph{the
  minimal annihilating polynomial} of $\bm{v}$ for the matrix $A$.
Furthermore, let $\pi_{A,j}(\lambda)$ denote, for $j=1,2,\ldots,n$,
the $j$-th \emph{unit minimal annihilating polynomial}
$\pi_{A,\bm{e}_j}(\lambda)$ of the $j$-th unit vector
$\bm{e}_j=(0,\ldots,0,1,0,\ldots,0)$. 

Assume that the irreducible factorization
\begin{equation}
  \label{eq:charpol}
  \chi_A(\lambda)=f_1(\lambda)^{m_1}f_2(\lambda)^{m_2}\cdots
  f_q(\lambda)^{m_q}
\end{equation}
over $\K$ of the characteristic polynomial $\chi_A(\lambda)$ is known.


Let $d_p=\deg(f_p)$ and set
\[
G_p = g_p(A), \quad F_p = f_p(A), \qquad\qquad p=1,2,\ldots,q,
\]
where
\begin{equation}
  \label{eq:gp}
  \begin{split}
    g_p(\lambda) &= f_1(\lambda)^{m_1}f_2(\lambda)^{m_2}\cdots
    f_{p-1}(\lambda)^{m_{p-1}}f_{p+1}(\lambda)^{m_{p+1}}\cdots
    f_q(\lambda)^{m_q}
    \\
    &= \chi_A(\lambda)/(f_p(\lambda)^{m_p}).
  \end{split}
\end{equation}

Now, let 
\begin{equation}
  \label{eq:jthunitminannihpol}
  \pi_{A,j}(\lambda)=f_1(\lambda)^{r_{j,1}}f_2(\lambda)^{r_{j,2}}\cdots
  f_q(\lambda)^{r_{j,q}}
\end{equation}
be the factorization of the $j$-th unit minimal annihilating
polynomial $\pi_{A,j}(\lambda)$, where $0\le r_{j,p}\leq m_p$ for
$1\leq p\leq q.$ 

Then, the exponent $r_{j,p}$ of the factor $f_p(\lambda)$ in
$\pi_{A,j}(\lambda)$ is identified as the minimum number $k$
satisfying $F_p^kG_p\bm{e}_j=\bm{0}$.  The simple observation above
leads a naive algorithm (Algorithm~\ref{alg:umap}) for calculating the
unit minimal annihilating polynomial(s).

Here we recall a vector version of the Horner's rule.

\begin{prop}
  \label{prop:matrix-vector-horner}
  Let $f(\lambda)\in\K[\lambda]$ be
  \begin{equation}
    \label{eq:f}
    f(\lambda)=a_d\lambda^d+a_{d-1}\lambda^{d-1}+\cdots+a_0\lambda^0,
  \end{equation}
  with $a_d\ne 0$, $A\in\K^{n\times n}$ and $\bm{v}\in\K^n$ be a
  column vector. Then, a vector $f(A)\bm{v}$ is calculated with
  $O(n^2d)$ arithmetic operations in $\K$.
\end{prop}
\begin{proof}
  $f(A)\bm{v}$ is calculated with the Horner's rule incorporating with 
  multiplication of $\bm{v}$ from the right as
  \begin{equation}
    \label{eq:fav}
    \begin{split}
      f(A)\bm{v} &= (a_dA^d+a_{d-1}A^{d-1}+\cdots+a_0E)\bm{v} \\
      &= A(\cdots A(A(a_d(A\bm{v}) +
      a_{d-1}\bm{v})+a_{d-2}\bm{v})\cdots)+a_0\bm{v},
    \end{split}
  \end{equation}
  with repeating pairs of a matrix-vector multiplication and a vector
  addition, whose complexity is $O(n^2)$ and $O(n)$, respectively, for
  $O(d)$ times. The cost is bounded by $O(n^2d)$ in total.
\end{proof}

For a row vector $\bm{w}\in\K^n$, $\bm{w}f(A)$ is calculated in the
same manner as in Proposition~\ref{prop:matrix-vector-horner}. We summarize
Proposition~\ref{prop:matrix-vector-horner}
as in Algorithms~\ref{alg:matrix-vector-horner} and
\ref{alg:vector-matrix-horner}
for use in other
algorithms in this paper.
(Please see \cite{taj-oha-ter2014} for a matrix version of the
Horner's rule.)

\begin{algorithm}[h]
  \caption{Calculating the $j$-th unit minimal annihilating
    polynomial $\pi_{A,j}(\lambda)$}  
  \label{alg:umap}
  \begin{algorithmic}[1]
    \Input{
      $A\in \K^{n\times n}$; \Comment{Input matrix;}

      $\chi_A(\lambda)=f_1(\lambda)^{m_1}f_2(\lambda)^{m_2}\cdots
      f_q(\lambda)^{m_q}\in \K[\lambda]$; \Comment{Irreducible
        factorization of the characteristic polynomial of $A$ as in
        \cref{eq:charpol};} 
    }
    \Output{\{$r_{j,1},\ldots,r_{j,q}$\};}\Comment{The list of
      exponents of the factors in $\pi_{A,j}(\lambda)$ as in
      \cref{eq:jthunitminannihpol}} 
    \For{$i=1,\ldots,q$}\label{alg:umap:for}
    \State{$g_i(\lambda)\gets \chi_A(\lambda)/(f_i(\lambda)^{m_i})$;}\label{alg:umap:gilambda}
    \State{$\bm{b}_{i,j}\gets
      \mvh(g_i(\lambda),A,\bm{e}_j$);}\label{alg:umap:bij} \Comment{$\bm{b}_{i,j}\gets g_i(A)\bm{e}_j$}
    \State{$k\gets 0$;}
    \While{$\bm{b}_{i,j}\ne\bm{0}$}
    \State{$\bm{b}_{i,j}\gets\mvh(f_i(\lambda),A,\bm{b}_{i,j})$;}
    \label{alg:umap:matrixvector} \Comment{$\bm{b}_{i,j}\gets f_i(A)\bm{b}_{i,j}$}
    \State{$k\gets k+1$;}
      \EndWhile
      \State{$r_{j,i}\gets k$;}
    \EndFor
    \State \textbf{return} \{$r_{j,1},\ldots,r_{j,q}$\}.
  \end{algorithmic}
\end{algorithm}

\newpage

\begin{rem}
  Assume that the minimal polynomial $\pi_A(\lambda)$ of $A$ together 
  with its irreducible factorization 
  \begin{equation}
    \label{eq:minpol}
    \pi_A(\lambda)=f_1(\lambda)^{l_1}f_2(\lambda)^{l_2}\cdots
    f_q(\lambda)^{l_q}
  \end{equation}
  is known.  Then, $m_p$ 
  can be replaced with $l_p$.  (Efficient algorithms for calculating
  the minimal polynomial (e.g.\ \cite{sto1998}) have been proposed.)
\end{rem}

In this paper, time complexity of algorithms is estimated with
arithmetic operations in $\K$, assuming that the
irreducible factorization of $\chi_A(\lambda)$ is given unless
otherwise stated.



\begin{algorithm}[p]
  \caption{The Horner's rule for matrix polynomial multiplied by a
    column vector from the right side}
  \label{alg:matrix-vector-horner}
  \begin{algorithmic}[1]
    \Input{
      $f(\lambda)=a_d\lambda^d+a_{d-1}\lambda^{d-1}+\cdots+a_0\lambda^0\in\K[\lambda]$; 

      $A\in \K^{n\times n}$;

      $\bm{v}\in\K^n$; \Comment{A column vector}

    }
    \Output{$f(A)\bm{v}$;}
    \Function{Matrix\_vector\_horner}{$f(\lambda)$, $A$, $\bm{v}$}
    \State \textbf{return} $f(A)\bm{v}$ calculated as in \cref{eq:fav}.
    \EndFunction
  \end{algorithmic}
\end{algorithm}

\begin{algorithm}
  \caption{The Horner's rule for matrix polynomial multiplied by a
    row vector from the left side}
  \label{alg:vector-matrix-horner}
  \begin{algorithmic}[1]
    \Input{
      $f(\lambda)=a_d\lambda^m+a_{d-1}\lambda^{m-1}+\cdots+a_0\lambda^0\in\K[\lambda]$;

      $A\in \K^{n\times n}$;

      $\bm{w}\in\K^n$; \Comment{A row vector}

    }
    \Output{$\bm{w}f(A)$;}
    \Function{Vector\_matrix\_horner}{$f(\lambda)$, $A$, $\bm{w}$}
    \State \textbf{return} $\bm{w}f(A)$ calculated in the
    same manner as in Proposition~\ref{prop:matrix-vector-horner}.
    \EndFunction
  \end{algorithmic}
\end{algorithm}


\begin{prop}
  \label{prop:umap:time}
  For given matrix $A\in \K^{n\times n}$ and irreducible factorization
  of its characteristic polynomial $\chi_A(\lambda)$,
  Algorithm~\ref{alg:umap} calculates the $j$-th minimal annihilating
  polynomial $\pi_{A,j}(\lambda)$ with
  \begin{equation}
    \label{eq:prop:umap:total}
    O\left(
      (q-1)n^3+n^2\deg(\pi_{A,j}(\lambda))
    \right)
  \end{equation}
  arithmetic operations in $\K$.
\end{prop}
\begin{proof}
  First, since $\deg(g_i(A))=n-d_im_i$, calculation of
  $\bm{b}_{i,j}=g_i(A)\bm{e}_j$ in line \ref{alg:umap:bij} takes
  $O(n^2(n-d_im_i))$ operations by
  Proposition~\ref{prop:matrix-vector-horner}.
  Repeating this calculation for
  $i=1,\ldots,q$ 
  takes $O((q-1)n^3)$ operations.
  Next, calculating $f_i(A)^{r_{j,i}}\bm{b}_{i,j}$ takes
  $O(r_{j,i}n^2d_i)$ operations by
  Proposition~\ref{prop:matrix-vector-horner}.  Repeating this
  calculation for $i=1,\ldots,q$ takes
  $O(n^2\deg(\pi_{A,j}(\lambda)))$ operations. Thus, the total number
  of operations becomes as in \cref{eq:prop:umap:total}, which proves
  the proposition.
\end{proof}


If we calculate the minimal annihilating polynomial
$\pi_{A,j}(\lambda)$ of the unit vector $\bm{e}_j$ for all
$j=1,\ldots,n$ by using the naive Algorithm~\ref{alg:umap}, the cost
of calculation is equal to
\begin{equation}
  \label{eq:prop:umap:total-all}
  O\left(
    (q-1)n^4+n^2\sum_{j=1}^n\deg(\pi_{A,j}(\lambda))
  \right).
\end{equation}
The first term give notice that calculation of $\bm{b}_{i,j}$'s is
time-consuming.
To overcome this issue, we introduce
\emph{pseudo} annihilating polynomials in the next section.

\section{Pseudo annihilating polynomials for calculating minimal
  annihilating polynomials}
\label{sec:pmap}

In this section, we introduce the notion of pseudo annihilating
polynomials of unit vectors in a constructive way and give an
algorithm for calculating minimal annihilating polynomials of all unit
vectors. 



Let $\bm{u}$ be a non-zero row vector over $\K$. Let
\begin{equation}
  \label{eq:wpk}
  \begin{split}
    \bm{w}_p^{(0)} &=
    (w_{p,1}^{(0)},w_{p,2}^{(0)},\ldots,w_{p,n}^{(0)})=\bm{u}G_p,  
    \\
    \bm{w}_p^{(k)} &=
    (w_{p,1}^{(k)},w_{p,2}^{(k)},\ldots,w_{p,n}^{(k)})=\bm{u}G_p{F_p}^k
    \quad
    \textrm{for $k>0$,}
  \end{split}
\end{equation}
where $G_p$ and $F_p$ are defined as in \cref{eq:gp}.

Furthermore, for $j=1,\ldots,n$, define
\begin{equation}
  \label{eq:rhopj}
  \rho_{p,j}=
  \begin{cases}
    0 & \text{if $w_{p,j}^{(0)}=0$,}  \\
    k & \text{if $w_{p,j}^{(k-1)}\ne 0$ and $w_{p,j}^{(k)}=0$.}
  \end{cases}
\end{equation}

Now we consider the polynomial $\pi'_{A,j,\bm{u}}(\lambda)$ defined by
\[
\pi'_{A,j,\bm{u}}(\lambda)=f_1^{\rho_{1,j}}(\lambda)f_2^{\rho_{2,j}}(\lambda)\cdots
f_q^{\rho_{q,j}}(\lambda).
\]

Then, we have the following lemma.
\begin{lem}
  \label{lem:rjp}
  The polynomial $\pi'_{A,j,\bm{u}}(\lambda)$ divides $\pi_{A,j}(\lambda)$.
  %
\end{lem}
\begin{proof}
  Since
  $\bm{u}G_pF_p^{k}\bm{e}_j=\bm{w}_p^{(k)}\bm{e}_j=w_{p,j}^{(k)}$,
  $r_{j,p}\ge \rho_{p,j}$ holds. This completes the proof.
\end{proof}

We call the polynomial $\pi'_{A,j,\bm{u}}(\lambda)$, pseudo
annihilating polynomial of the unit vector $\bm{e}_j$ for the matrix
$A$ with respect to the row vector $\bm{u}$, or the $j$-th unit pseudo
annihilating polynomial.

Note that if the vector $\bm{u}$ above is sufficiently random, we have 
$\pi'_{A,j,\bm{u}}(\lambda)=\pi_{A,j}(\lambda)$.

We summarize the steps for calculating the unit pseudo annihilating
polynomials in Algorithm~\ref{alg:upap}. Let us write hereafter the output of
Algorithm~\ref{alg:upap}, as $\pi'_{A,j}(\lambda)$ for short. Since
the row vector $\bm{u}$ of the pseudo annihilating 
polynomial $\pi'_{A,j,\bm{u}}(\lambda)$ of the unit vector $\bm{e}_j$
is a random vector generated in the algorithm, the existence of a row
vector, say $\bm{u}$, used to calculate
$\pi'_{A,j,\bm{u}}(\lambda)$ is known, whereas the vector $\bm{u}$ is
not known.  We call $\pi'_{A,j}(\lambda)$ a pseudo annihilating
polynomial of $\bm{e}_j$, or a $j$-th unit pseudo annihilating
polynomial. 
\begin{algorithm}
  \caption{Calculating unit pseudo annihilating polynomials
    $\pi'_{A,j}(\lambda)$}
  \label{alg:upap}
  \begin{algorithmic}[1]
    \Input{
      $A\in K^{n\times n}$; \Comment{Input matrix;}

      $\chi_A(\lambda)=f_1(\lambda)^{m_1}f_2(\lambda)^{m_2}\cdots
      f_q(\lambda)^{m_q}\in K[\lambda]$; \Comment{Irreducible
        factorization of the characteristic polynomial of $A$;}
    }
    \Output{$P=(\rho_{i,j})\in\R^{q\times n}$, where $\rho_{i,j}$ is
      equal to exponent of factor $f_i(\lambda)$ in a $j$-th unit
      pseudo annihilating polynomial $\pi'_{A,j}(\lambda)$;} 
    \Function{Unit\_pseudo\_annih}{$A$, $\{\{f_1(\lambda), m_1\},
      \{f_2(\lambda), m_2\}, \dots, \{f_q(\lambda), m_q\}\}$}
    \State{$\bm{u}\gets $ (a randomly generated row vector of dimension
        $n$);}
    \For{$i=1,\ldots,q$}\label{alg:upap:for-i}
      \State{$g_i(\lambda)\gets
        \chi_A(\lambda)/(f_i(\lambda)^{m_i})$;}
      \label{alg:upap:gilambda}
      \State{$\bar{\bm{b}}_i=(\bar{b}_i^{(1)},\ldots,\bar{b}_i^{(n)})
        \gets \vmh(g_i(\lambda),A,\bm{u})$;} \label{alg:upap:ugia}
      \Comment{$\bar{\bm{b}}_i\gets\bm{u}g_i(A)$}
      \State{$k\gets 0$;}
      \State{$\textbf{FL}=(\fl_1,\fl_2,\ldots,\fl_n)=(0,\ldots,0)$;}
      \Comment{A tuple of flags such that $\fl_j=1$ if exponent
        $\rho_{i,j}$ of $f_i(\lambda)$ in a $j$-th unit pseudo
        annihilating polynomial $\pi'_{A,j}(\lambda)$ has been
        discovered;} 
      \For{$l=1,\ldots,m_i$} \label{alg:upap:for-l}
        \If{$\bar{\bm{b}}_i=\bm{0}$ and $\textbf{FL}=(1,\dots,1)$}
          \State break;
        \EndIf
        \For{$j=1,\ldots,n$}
        \If{$\bar{b}_i^{(j)}=0$ and $\textrm{FL}_j=0$}
        \Comment{Exponent of $f_i(\lambda)$ in
          $\pi'_{A,j}(\lambda)$ found;} 
          \State{$\rho_{i,j}\gets k$;}
          \State{$\fl_j\gets 1$;}
        \EndIf
        \EndFor
        \State{$\bar{\bm{b}}_i\gets\vmh(f_i(\lambda),A,\bar{\bm{b}}_i)$;} \label{alg:upap:vectormatrix}
        \Comment{$\bar{\bm{b}}_i\gets\bar{\bm{b}}_i f_i(A)$}
        \State{$k\gets k+1$;}
      \EndFor
      \algstore{alg:upap}
    \end{algorithmic}
  \end{algorithm}
  \begin{algorithm}
    \ContinuedFloat
    \caption{Calculating unit pseudo annihilating polynomials
      $\pi'_{A,j}(\lambda)$ (Continued)}
    \begin{algorithmic}
      \algrestore{alg:upap}
      \For{$j=1,\ldots,n$} \label{alg:upap:for-j}
        \If{$\textrm{FL}_j = 0$} \Comment{For $f_i(\lambda)$ of whose
          exponent in $\pi'_{A,j}(\lambda)$ has not found, it
          should be $m_i$;}
          \State{$\rho_{i,j}\gets m_i$;}
        \EndIf
      \EndFor
    \EndFor
    \State \Return $P=(\rho_{i,j});$
    \Comment{$\pi'_{A,j}(\lambda)=
        f_1(\lambda)^{\rho_{1,j}}f_2(\lambda)^{\rho_{2,j}}\cdots  
        f_q(\lambda)^{\rho_{q,j}}$;}
      \EndFunction
  \end{algorithmic}
\end{algorithm}

\begin{prop}
  \label{prop:upap-time}
  Algorithm~\ref{alg:upap} outputs
  \emph{all} unit pseudo annihilating polynomials for the matrix $A$
  with 
  \begin{equation}
    \label{eq:upap-time}
    O\left(
      (q-1)n^3+n^2\deg(\pi_A(\lambda))
    \right)
  \end{equation}
  arithmetic operations in $\K$.
  \begin{proof}

    We estimate time complexity of the algorithm as follows.  First,
    the amount of operations required for calculating $\bar{\bm{b}}_i$
    in line
    \ref{alg:upap:ugia} is
    estimated $O(n^2(n-d_im_i))$.
    Repeating this calculation for
    $i=1,\ldots,q$
    requires
    \begin{equation}
      \label{eq:prop:upap:1stterm}
      O((q-1)n^3)
    \end{equation}
    arithmetic operations.

    Next, in line
    \ref{alg:upap:vectormatrix}, calculating $\bm{\bar{b}}_if_i(A)$
    requires $O(n^2d_i)$ operations.  For each $i$ in the ``for'' loop in
    line \ref{alg:upap:for-i}, the ``for'' loop in line
    \ref{alg:upap:for-l} repeats for
    $\max_{j\in\{1,\dots,n\}}\rho_{i,j}$. Since
    $\rho_{i,j}\le r_{j,i}$ and
    $\max_{j\in\{1,\dots,n\}}r_{j,i}=l_i$,
    the number of operation is bounded by
    \begin{equation}
      \label{eq:prop:upap:2ndterm}
      O
      \left(
        n^2\sum_{i=1}^q(l_id_i)
      \right)
      =O(n^2\deg(\pi_A(\lambda)))
    \end{equation}
    for $i=1,\ldots,q$ in the ``for'' loop in line
    \ref{alg:upap:for-i}.
    Sum of the amounts in
    \cref{eq:prop:upap:1stterm,eq:prop:upap:2ndterm} gives an estimate
    of the number of operations in the whole algorithm as in
    \cref{eq:upap-time}, which proves the proposition.
  \end{proof}
\end{prop}


\begin{rem}
  \label{rem:prop:upap:minpoly}
  If we have the minimal polynomial together with its irreducible
  factorization as in \cref{eq:minpol}, the computing time is
  \[
  O\left(
    qn^2(\deg(\pi_A(\lambda)))
  \right).
  \]
\end{rem}

\begin{rem}
  \label{rem:prop:upap:parallel}
  In Algorithm~\ref{alg:upap}, each processes in the ``for'' loop in
  line~\ref{alg:upap:for-i} is independent each other. This step can
  be executed in parallel. For
  example, if we distribute each processes to $M$ processors
  ($M\le q$), the
  computing time in \cref{eq:upap-time} is
  $O\left(
    \frac{q-1}{M}n^3+\frac{n^2}{M}\deg(\pi_A(\lambda))
  \right)$.
\end{rem}

Algorithm~\ref{alg:upap} is used in the next
Algorithm~\ref{alg:umap-p} for calculating the unit annihilating
polynomials.
\begin{algorithm}
  \caption{Calculating the unit minimal annihilating polynomials
    $\pi_{A,j}(\lambda)$}
  \label{alg:umap-p}
  \begin{algorithmic}[1]
    \Input{
      $A\in K^{n\times n}$; \Comment{Input matrix;}

      $\chi_A(\lambda)=f_1(\lambda)^{m_1}f_2(\lambda)^{m_2}\cdots
      f_q(\lambda)^{m_q}\in K[\lambda]$; \Comment{Irreducible
        factorization of the characteristic polynomial of $A$;}
    }
    \Output{$R=(r_{i,j})\in\R^{q\times n}$, where $r_{i,j}$ is
      equal to exponent of factor $f_i(\lambda)$ in the $j$-th unit
      minimal annihilating polynomial $\pi_{A,j}(\lambda)$;} 

    \Function{Unit\_minimal\_Annih}{$A$, $\{\{f_1(\lambda), m_1\},
      \{f_2(\lambda), m_2\}, \dots, \{f_q(\lambda), m_q\}\}$}
    \State $P\gets$\Call{Unit\_pseudo\_annih}{$A$,
      $\{\{f_1(\lambda),m_1\},\{f_2(\lambda), m_2\}, \dots,
      \{f_q(\lambda), m_q\}\}$}; \label{alg:umap-p-get-upap}
    \Comment{$P=(\rho_{i,j})$;}
    \State{$R=(r_{i,j})\gets P$;}
    \For{$j=1,\ldots,n$} \label{alg:umap-p-for-j}
      \State{$\pi'_{A,j}(\lambda)\gets f_1(\lambda)^{\rho_{1,j}}
        \cdots f_q(\lambda)^{\rho_{q,j}}$;} \label{alg:umap-p-call-upap}
      \Comment{A $j$-th unit pseudo annihilating polynomial for the
        matrix $A$;} 
      \State{$\bm{v}\gets\mvh(\pi'_{A,j}(\lambda),A,\bm{e}_j)$;} \label{alg:test-puap}
      \Comment{$\bm{v}\gets\pi'_{A,j}(A)\bm{e}_j$;}
      \If{$\bm{v}=\bm{0}$} \label{alg:test-puap-true1}
        \State{continue;}
      \Else{}
      \State{$\bm{v}_0\gets\bm{v}$;}
      \For{$i=1,\ldots,q$} \label{alg:umap-p-i-loop}
        \For{$l=1,\ldots,m_i-\rho_{i,j}$} \label{alg:umap-p-l-loop}
          \State{$\bm{v}\gets\mvh(f_i(\lambda),A,\bm{v})$;} \label{alg:umap-p-get-v}
          \Comment{$\bm{v}\gets f_i(A)\bm{v}$}
          \If{$\bm{v}=\bm{0}$} \label{alg:test-puap-true2}
            \State{$r_{i,j}\gets \rho_{i,j}+l$;} \Comment{Exponent of
              $f_i$ in a $j$-th unit minimal annihilating polynomial
              is equal to $\rho_{i,j}+l$;}
            \label{alg:umap-p-set-rij-for-the-largest}
            \State{$\bm{v}_i\gets\bm{0}$;}
            \For{$k=i-1,\ldots,0$} \label{alg:umap-p-update-vk-begin}
              \State{$\bm{v}_k\gets
                \mvh(f_i(\lambda)^l,A,\bm{v}_k)$;} \label{alg:umap-p-update-vk}
              \Comment{$\bm{v}_k\gets f_i(A)^l\bm{v}_k$}
            \EndFor \label{alg:umap-p-update-vk-end}
        \algstore{alg:umap-p}
      \end{algorithmic}
    \end{algorithm}
    \begin{algorithm}
      \ContinuedFloat
      \caption{Calculating the unit minimal annihilating polynomials
        $\pi_{A,j}(\lambda)$ (Continued)}
      \begin{algorithmic}
        \algrestore{alg:umap-p}
            \State{break;}
          \EndIf
        \EndFor
        \If{$\bm{v}=\bm{0}$}
          \State{break;} \label{alg:umap-p-break-i-loop}
        \Else
          \State{$\bm{v}_i\gets\bm{v}$;}
        \EndIf
      \EndFor \label{alg:umap-p-i-end}
      \For{$k=i-1,\ldots,1$} \label{alg:umap-p-k}
        \State{$\bm{v}\gets\bm{v}_{k-1}$;} \label{alg:umap-p-k-vk-1}
        \For{$l=0,\dots,m_k-\rho_{k,j}-1$} \label{alg:umap-p-k-l}
          \If{$\bm{v}=\bm{0}$} \label{alg:umap-p-k-vk-1-v}
            \For{$s=k-2,\ldots,0$} \label{alg:umap-p-update-vk-1-begin}
              \State{$\bm{v}_s\gets
                \mvh(f_k(\lambda)^l,A,\bm{v}_s)$;} \label{alg:umap-p-update-vs}
              \Comment{$\bm{v}_s\gets f_k(A)^l\bm{v}_s$}
            \EndFor \label{alg:umap-p-update-vk-1-end}
            \State{break;}
          \EndIf
          \State{$\bm{v}\gets
            \mvh(f_k(\lambda),A,\bm{v})$;} \label{alg:umap-p-multiply-v-by-fk}
          \Comment{$\bm{v}\gets f_k(A)\bm{v}$}
          \State{$r_{k,j}\gets r_{k,j} +1$;}
        \EndFor
      \EndFor \label{alg:umap-p-k-end}
      \EndIf
    \EndFor
    \State \Return $R=(r_{i,j});$
    \Comment{$\pi_{A,j}(\lambda)=
      f_1(\lambda)^{r_{1,j}}f_2(\lambda)^{r_{2,j}}\cdots  
      f_q(\lambda)^{r_{q,j}}$;}
    \EndFunction
  \end{algorithmic}
\end{algorithm}
\newpage

\begin{rem}
  In Algorithm~\ref{alg:umap-p}, in the case $\bm{v}=\bm{0}$ in
  line~\ref{alg:test-puap}, then the pseudo annihilating polynomial
  calculated in line~\ref{alg:umap-p-call-upap} is true unit minimal
  annihilating polynomial, thus no more calculation is needed. On the
  other hand, in the case $\bm{v}\ne\bm{0}$,
  the pseudo
  annihilating polynomial is a factor of true unit minimal
  annihilating polynomial and $\bm{v}$ is a partial result, thus
  calculation of the unit minimal annihilating
  polynomial is accomplished by calculating the minimal annihilating
  polynomial of $\bm{v}$ as in the rest of the algorithm.
  In this way, true unit annihilating polynomial is derived from the
  pseudo unit annihilating polynomial without restarting whole
  calculation, that makes proposed algorithm very efficient.

\end{rem}

We show the validity and the time complexity of the algorithm by
the following propositions.

\begin{prop}
  \label{prop:alg:umap-p}
  Algorithm~\ref{alg:umap-p} outputs
  the exponents of factors in the unit annihilating polynomials of
  $A$.
\end{prop}
\begin{proof}
  Recall that
  \begin{equation}
    \label{eq:unitannihpol}
    \pi_{A,j}(\lambda)=
      f_1(\lambda)^{r_{1,j}}f_2(\lambda)^{r_{2,j}}\cdots  
      f_q(\lambda)^{r_{q,j}},
  \end{equation}
  and
  \begin{equation}
    \label{eq:pseudounitannihpol}
    \pi'_{A,j}(\lambda)=
    f_1(\lambda)^{\rho_{1,j}}f_2(\lambda)^{\rho_{2,j}}\cdots  
    f_q(\lambda)^{\rho_{q,j}},
  \end{equation}
  a $j$-th unit pseudo annihilating polynomial calculated by
  Algorithm~\ref{alg:upap}.
  We consider the following cases according to lines
  \ref{alg:test-puap} and \ref{alg:test-puap-true1}.
  
  \textbf{Case 1}:
  $\bm{v}=\bm{0}$. In this case, the
  algorithm outputs $\pi'_{A,j}(\lambda)$ as
  $\pi_{A,j}(\lambda)$. 

  \textbf{Case 2}: $\bm{v}\ne\bm{0}$. For
  $j=1,\ldots,n$, $r_{i,j}$ in \cref{eq:unitannihpol} and
  $\rho_{i,j}$ in \cref{eq:pseudounitannihpol}, let
  \begin{equation}
    \label{eq:umap-p-q'j}
    q''_j = \max\{1\le i\le q\mid \rho_{i,j}<r_{i,j}\},
  \end{equation}
  and let $\delta_{i,j}=r_{i,j}-\rho_{i,j}$.

  For every $i$ in the ``for'' loop from line \ref{alg:umap-p-i-loop}
  and $l$ at the beginning of the ``for'' loop in
  line~\ref{alg:umap-p-l-loop}, we have
  \[
  \bm{v}=f_1(A)^{m_1}\cdots f_{i-1}(A)^{m_{i-1}} f_i(A)^{\rho_{i,j}+l-1}
  f_{i+1}(A)^{\rho_{i+1,j}}\cdots f_{q}(A)^{\rho_{q,j}}\bm{e}_j.
  \]
  For making $\bm{v}=\bm{0}$ in line \ref{alg:test-puap-true2},
  exponent of $f_i(A)$ must be greater than or equal to $r_{i,j}$ for
  $i=1,\ldots,q''_j$. In fact, for the first time when the condition
  at the line
  \ref{alg:test-puap-true2} is satisfied, we have $i=q''_j$,
  $l=r_{q''_j,j}-\rho_{q''_j,j}$, and
  \begin{multline}
    \bm{v}=\bm{0}
    \\
    =f_1(A)^{m_1}\cdots f_{q''_j-1}(A)^{m_{q''_j-1}}
    f_{q''_j}(A)^{r_{q''_j,j}}f_{q''_j+1}(A)^{\rho_{q''_j+1,j}}\cdots
    f_{q}(A)^{\rho_{q,j}}\bm{e}_j. \label{eq:puap-update-satisfied}
  \end{multline}
  Then, by line \ref{alg:umap-p-set-rij-for-the-largest}, we have
  $r_{q''_j,j}\gets \rho_{q''_j,j}+(r_{q''_j,j}-\rho_{q''_j,j})=r_{q''_j,j}$.
  
  At the end of ``for'' loop in line \ref{alg:umap-p-i-end} for the
  $i$-th time, we have
  \begin{equation}
    \label{eq:puap-vs}
    \bm{v}_s=f_1(A)^{m_1}\cdots
    f_{s}(A)^{m_{s}}f_{s+1}(A)^{\rho_{s+1,j}}\cdots
    f_{q}(A)^{\rho_{q,j}}\bm{e}_j
  \end{equation}
  for $s=0,\ldots,i$ (note that we do not have factors
  $f_1(A)^{m_1}\cdots f_{s}(A)^{m_s}$ for $s=0$). Thus, when the
  condition at the line
  \ref{alg:test-puap-true2} is satisfied for $i=q''_j$, we have
  \cref{eq:puap-vs} for $s=0,\ldots,q''_j-1$. Then, by ``for'' loop
  between line \ref{alg:umap-p-update-vk-begin} and
  \ref{alg:umap-p-update-vk-end}, $\bm{v}_s$ in \cref{eq:puap-vs} gets
  updated as
  \begin{multline*}
    \bm{v}_s=f_1(A)^{m_1}\cdots
    f_{s}(A)^{m_{s}}f_{s+1}(A)^{\rho_{s+1,j}}\cdots
    f_{q''_j-1}(A)^{\rho_{q''_j-1,j}}
    \\
    \times f_{q''_j}(A)^{r_{q''_j,j}}
    f_{q''_j+1}(A)^{\rho_{q''_j+1,j}}\cdots
    f_{q}(A)^{\rho_{q,j}}\bm{e}_j,
  \end{multline*}
  for $s=0,\ldots,q''_j-1$ (note that exponent of $f_{q''_j}(A)$ is equal to
  $r_{q''_j,j}$ which is equal to the one in the unit annihilating
  polynomial). 

  After exiting from ``for'' loop at line \ref{alg:umap-p-i-end}, we
  have $i=q''_j$, thus, at the first time for ``for'' loop in line
  \ref{alg:umap-p-k}, we have $k=q''_j-1$.
  For $k=q''_j-1,\ldots,1$ in the ``for'' loop from line
  \ref{alg:umap-p-k}, we have
  \begin{multline*}
    \bm{v}=\bm{v}_{k-1}=f_1(A)^{m_1}\cdots
    f_{k-1}(A)^{m_{k-1}}
    f_{k}(A)^{\rho_{k,j}} 
    \\
    \times
    f_{k+1}(A)^{r_{k+1,j}} \cdots
    f_{q''_j}(A)^{r_{q''_j,j}}
    f_{q''_j+1}(A)^{\rho_{q''_j+1,j}} \cdots
    f_{q}(A)^{\rho_{q,j}}\bm{e}_j.
  \end{multline*}
  If $\bm{v}$ satisfies the condition in line
  \ref{alg:umap-p-k-vk-1-v}, then
  \begin{multline}
    \label{eq:umap-k-vm-1-v}
    \bm{v}=\bm{0}=f_1(A)^{m_1}\cdots
    f_{k-1}(A)^{m_{k-1}}f_{k}(A)^{r_{k,j}}
    \\
    \times
    f_{k+1}(A)^{r_{k+1,j}} \cdots
    f_{q''_j}(A)^{r_{q''_j,j}} 
    f_{q''_j+1}(A)^{\rho_{q''_j+1,j}}\cdots
    f_{q}(A)^{\rho_{q,j}}\bm{e}_j,
  \end{multline}
  with $l=r_{k,j}-\rho_{k,j}$. Thus, by ``for'' loop between line
  \ref{alg:umap-p-update-vk-1-begin} and
  \ref{alg:umap-p-update-vk-1-end}, $\bm{v}_s$ in \cref{eq:puap-vs}
  gets updated as
  \begin{multline*}
    \bm{v}_s=f_1(A)^{m_1}\cdots
    f_{s}(A)^{m_{s}}f_{s+1}(A)^{\rho_{s+1,j}}\cdots
    f_{k-1}(A)^{\rho_{k-1,j}}
    \\
    \times
    f_{k}(A)^{r_{k,j}} \cdots
    f_{q''_j}(A)^{r_{q''_j,j}}
    f_{q''_j+1}(A)^{\rho_{q''_j+1,j}}\cdots
    f_{q}(A)^{\rho_{q,j}}\bm{e}_j,
  \end{multline*}
  for $s=0,\ldots,k-1$. Finally, for $k=1$, we have
  \cref{eq:puap-vs} as
  \begin{equation*}
    \bm{v}_0=
    f_{1}(A)^{r_{1,j}}
    f_{2}(A)^{r_{2,j}} \cdots
    f_{q''_j}(A)^{r_{q''_j,j}} 
    f_{q''_j+1}(A)^{\rho_{q''_j+1,j}}\cdots
    f_{q}(A)^{\rho_{q,j}}\bm{e}_j,
  \end{equation*}
  with $l=r_{1,j}-\rho_{1,j}$.  Thus, at exiting line
  \ref{alg:umap-p-k-end}, we have $r_{i,j}$ satisfying
  $\pi_{A,j}(\lambda)=
  f_1(\lambda)^{r_{1,j}}f_2(\lambda)^{r_{2,j}}\cdots
  f_q(\lambda)^{r_{q,j}}$ for each $j$.
\end{proof}

\begin{prop}
  \label{prop:alg:umap-p-time}
  Algorithm~\ref{alg:umap-p} outputs
  the result with
  \begin{multline}
    \label{eq:alg:umap-p-time}
    O\biggl(
    (q-1)n^3
    +n^2\deg(\pi_A(\lambda))
    +n^2\sum_{j=1}^n\deg\pi'_{A,j}(\lambda)
    \\
    +n^2\sum_{j=1}^{n}\sum_{k=1}^{q''_j}
    d_k\{(m_k-\rho_{kj})+k\delta_{k,j}\} \biggr)
  \end{multline}
  arithmetic operations in $\K$, where $q''_j$ is defined as in
  \cref{eq:umap-p-q'j} and
  $\delta_{i,j}=r_{i,j}-\rho_{i,j}$ with $j=1,\ldots,n$, $r_{i,j}$ as
  in \cref{eq:unitannihpol} and $\rho_{i,j}$ as in
  \cref{eq:pseudounitannihpol}.
\end{prop}
\begin{proof}
  First, note that, line
  \ref{alg:umap-p-get-upap} can be executed with
  \begin{equation}
    \label{eq:umap-p:time:upap}
    O\left(
      (q-1)n^3+n^2\deg(\pi_A(\lambda))
    \right)
  \end{equation}
  arithmetic operations as in \cref{eq:upap-time}.  Then, in the
  ``for'' loop in line \ref{alg:umap-p-for-j} for $j=1,\ldots,n$, each
  loop has the following operations on vectors and matrices.

  In line~\ref{alg:test-puap}, calculating $\bm{v}$ takes
  $O(n^2\deg\pi'_{A,j}(\lambda))$ operations. Thus, the number
  of arithmetic operations for this line is 
  \begin{equation}
    \label{eq:alg:test-puap:time}
    \sum_{j=1}^nO(n^2\deg\pi'_{A,j}(\lambda)),
  \end{equation}
  for the ``for'' loop in line \ref{alg:umap-p-for-j}.

  In line \ref{alg:umap-p-get-v}, calculating $\bm{v}$ takes
  $O(d_in^2)$ operations for each $i$ by the Horner's rule,
  as shown in Proposition~\ref{prop:matrix-vector-horner}. Since this
  line is called for $l=1,\ldots,m_i$ (in line
  \ref{alg:umap-p-l-loop}) with $i=1,\ldots,q''_j-1$ (in line
  \ref{alg:umap-p-i-loop}) and $l=1,\ldots,\delta_{i,j}$ with $i=q''_j$,
  the number of arithmetic operations for this line is bounded by
  \begin{equation}
    \label{eq:alg:umap-p-get-v-time}
    \sum_{i=1}^{q''_j}O(d_i(m_i-\rho_{i,j})n^2).
  \end{equation}

  In line \ref{alg:umap-p-update-vk}, calculating $f_i(A)^l\bm{v}_k$
  takes $O(d_{q''_j}\delta_{q''_j,j}n^2)$ operations. Since this line is
  called for $k=i-1,\ldots,0$ (in line
  \ref{alg:umap-p-update-vk-begin}) with $i=q''_j$ (since this line is
  called when $\bm{v}$ satisfies line \ref{alg:test-puap-true2} that
  occurs only once for $i=q''_j$; immediately after that there will be
  a break in line \ref{alg:umap-p-break-i-loop} of the ``for'' loop in
  $i$ (in line \ref{alg:umap-p-i-loop})), the number of arithmetic
  operations for this line is bounded by
  \begin{equation}
    \label{eq:alg:umap-p-update-vk-time}
    O(q''_jd_{q''_j}\delta_{q''_j,j}n^2).
  \end{equation}

  In line \ref{alg:umap-p-update-vs}, calculating $f_k(A)^l\bm{v}_s$
  takes $O(d_k\delta_{k,j}n^2)$ operations. Since this line is
  called for $k-1$ times (in the ``for'' loop in line
  \ref{alg:umap-p-update-vk-1-begin}) in each $k$, the number of
  arithmetic operations for this line is bounded by
  $O(kd_k\delta_{k,j}n^2)$ in that loop. Furthermore, this loop is
  called for $k=i-1,\ldots,1$ (in the ``for'' loop in line
  \ref{alg:umap-p-k}) with $i=q''_j$, the number of arithmetic
  operations for this line is bounded by
  \begin{equation}
    \label{eq:alg:umap-p-update-vs-time}
    \sum_{k=1}^{q''_j-1}O(kd_k\delta_{k,j}n^2).
  \end{equation}

  In line \ref{alg:umap-p-multiply-v-by-fk}, calculating
  $f_k(A)\bm{v}$ takes $O(d_kn^2)$ operations. Since this line is
  called for $\delta_{k,j}$ times (for
  $l=0,\dots,d_{k,j}-\rho_{k,j}-1$ in the ``for'' loop in line
  \ref{alg:umap-p-k-l}) in each $k$ and for $k=i-1,\ldots,1$ (in the
  ``for'' loop in line \ref{alg:umap-p-k}) with $i=q''_j$, the number
  of arithmetic operations for this line is bounded by
  \begin{equation}
    \label{eq:alg:umap-p-multiply-v-by-fk-time}
    \sum_{k=1}^{q''_j-1}O(d_k\delta_{k,j}n^2).
  \end{equation}

  In the above, we see that
  \cref{eq:alg:umap-p-update-vk-time,eq:alg:umap-p-update-vs-time}
  are combined as
  \begin{equation}
    \label{eq:alg:umap-p-update-vs-time-2}
    \sum_{k=1}^{q''_j}O(kd_k\delta_{k,j}n^2),
  \end{equation}
  which dominates \cref{eq:alg:umap-p-multiply-v-by-fk-time}.
  By
  \cref{eq:alg:umap-p-get-v-time,eq:alg:umap-p-update-vs-time-2},
  the number of arithmetic operations in the ``for''
  loop in line \ref{alg:umap-p-i-loop} is
  \begin{equation}
    \label{eq:alg:umap-p-i-loop-time}
    \sum_{k=1}^{q''_j}O(d_k\{(m_k-\rho_{kj})+k\delta_{k,j}\}n^2),
  \end{equation}
  which becomes
  \begin{equation}
    \label{eq:alg:umap-p-for-j-time}
    \sum_{j=1}^{n}\sum_{k=1}^{q''_j}
    O(d_k\{(m_k-\rho_{kj})+k\delta_{k,j}\}n^2),
  \end{equation}
  for $j=1,\ldots,n$ (as in the ``for'' loop in line
  \ref{alg:umap-p-for-j}). Finally, by adding the result in
  \cref{eq:alg:umap-p-for-j-time} together with the ones in
  \cref{eq:umap-p:time:upap,eq:alg:test-puap:time}, we have \cref{eq:alg:umap-p-time}, which
  proves the proposition.

\end{proof}

\begin{rem}
  \label{rem:prop:umap:minpoly}
  If we have the minimal polynomial
  $\pi_A(\lambda)$ together with its irreducible factorization as in
  \cref{eq:minpol}, the computing time is
  \begin{multline*}
    O\Biggl(
    qn^2(\deg(\pi_A(\lambda)))
    +n^2\sum_{j=1}^n\deg\pi'_{A,j}(\lambda)
    \\
    +n^2\sum_{j=1}^{n}\sum_{k=1}^{q''_j}
    d_k\{(l_k-\rho_{kj})+k\delta_{k,j}\} \Biggr),
  \end{multline*}
  (cf.\ \cref{eq:alg:umap-p-time}).
\end{rem}

\begin{rem}
  \label{rem:prop:umap:parallel}
  In Algorithm~\ref{alg:umap-p}, each processes in the ``for'' loop in
  line~\ref{alg:umap-p-for-j} is independent to each other.
  If we distribute each processes to $M$ processors, the
  computing time in \cref{eq:alg:umap-p-time}
  is
  \begin{multline*}
    O
    \Biggl(
    \frac{q-1}{\min\{M,q\}}n^3+\frac{n^2}{\min\{M,q\}}\deg(\pi_A(\lambda))
    +\frac{n^2}{M}\sum_{j=1}^n\deg\pi'_{A,j}(\lambda)
    \\
    +\frac{n^2}{M}\sum_{j=1}^{n}\sum_{k=1}^{q''_j}
    d_k\{(m_k-\rho_{kj})+k\delta_{k,j}\}
    \Biggr),
  \end{multline*}
  where first two terms are from Algorithm~\ref{alg:upap} and they have
  $\min\{M,q\}$ in the denominators.
\end{rem}



\section{Efficient calculation of unit pseudo annihilating polynomials} 
\label{sec:upap}

Now, we propose an efficient method for calculating unit pseudo
annihilating polynomials which is the main result of the present
paper.  For calculating
$\pi'_{A,j}(\lambda)$ for $j=1,2,\ldots,n$ in
Algorithm~\ref{alg:upap}, $q$ row vectors
$\bm{w}_1^{(0)}=\bm{u}G_1,\bm{w}_2^{(0)}=\bm{u}G_2,\ldots,\bm{w}_q^{(0)}=\bm{u}G_q$
of dimension $n$ are needed, which are defined as
\begin{equation}
  \label{eq:w-vectors}
  \begin{split}
    \bm{w}_1^{(0)} &= \bm{u}F_2^{m_2}F_3^{m_3}\cdots F_q^{m_q},\\
    \bm{w}_2^{(0)} &= \bm{u}F_1^{m_2}F_3^{m_3}\cdots F_q^{m_q},\\
    & \vdots \\
    \bm{w}_q^{(0)} &= \bm{u}F_1^{m_2}F_2^{m_2}\cdots F_{q-1}^{m_{q-1}},
  \end{split}
\end{equation}
where $F_i=f_i(A)$ for $i=1,\ldots,q$ and $\bm{u}$ is a random
vector. Since $G_1,\ldots,G_q$ consist of almost the same factors,
$\bm{w}_1^{(0)},\ldots,\bm{w}_q^{(0)}$ are calculated efficiently with
the use of
binary splitting technique, as follows.

\begin{expl}
  \label{expl:binary-tree}
  We show an example for $q=8$. 
  Let $\bm{u}_{(i_1,i_2,\ldots,i_k)}$ denote the vector
  $\bm{u}F_{i_1}^{m_{i_1}}F_{i_2}^{m_{i_2}}\cdots F_{i_k}^{m_{i_k}}$,
  for simplicity. For example,
  $\bm{w}_1^{(0)}=\bm{u}_{(2,3,4,5,6,7,8)}$.
  Consider for instance the four vectors
  $\bm{w}_1^{(0)},\bm{w}_2^{(0)},\bm{w}_3^{(0)},\bm{w}_4^{(0)}$.
  Since $\bm{w}_i^{(0)}=\bm{u}G_i$ and $G_1,G_2,G_3,G_4$ have
  $F_5^{m_5}F_6^{m_6}F_7^{m_7}F_8^{m_8}$ as a common factor, we
  first calculate
  $\bm{u}_{(5,6,7,8)}=\bm{u}F_5^{m_5}F_6^{m_6}F_7^{m_7}F_8^{m_8}$. Second,
  by multiplying $F_3^{m_3}F_4^{m_4}$ and $F_1^{m_1}F_2^{m_2}$, we
  have
  \begin{align}
    \bm{u}_{(3,4,5,6,7,8)} &= \bm{u}_{(5,6,7,8)}(F_3^{m_3}F_4^{m_4}),
    \label{eq:w12}
    \\
    \bm{u}_{(1,2,5,6,7,8)} &= \bm{u}_{(5,6,7,8)}(F_1^{m_1}F_2^{m_2}), 
    \label{eq:w34}
  \end{align}
  respectively. Lastly, by multiplying $\bm{u}_{(3,4,5,6,7,8)}$
  by $F_2^{m_2}$ and $F_1^{m_1}$ in \cref{eq:w12}, we have
  $\bm{u}_{(2,3,4,5,6,7,8)}=\bm{w}_1^{(0)}$ and
  $\bm{u}_{(1,3,4,5,6,7,8)}=\bm{w}_2^{(0)}$, respectively, and by
  multiplying $\bm{u}_{(1,2,5,6,7,8)}$ by $F_4^{m_4}$ and $F_3^{m_3}$
  in \cref{eq:w34}, we obtain
  $\bm{u}_{(1,2,4,5,6,7,8)}=\bm{w}_3^{(0)}$ and
  $\bm{u}_{(1,2,3,5,6,7,8)}=\bm{w}_4^{(0)}$, respectively.  With the
  same manner,
  $\bm{u}_{(1,2,3,4,6,7,8)}=\bm{w}_5^{(0)},\ldots,\bm{u}_{(1,2,3,4,5,6,7)}=\bm{w}_8^{(0)}$
  can be calculated
  as well, shown as a binary tree in
  Figure~\ref{fig:pseudo-annih-pol-binary-tree}.  In this example,
  $\{\bm{w}_1^{(0)},\bm{w}_2^{(0)},\bm{w}_3^{(0)},\bm{w}_4^{(0)}\}$
  and
  $\{\bm{w}_5^{(0)},\bm{w}_6^{(0)},\bm{w}_7^{(0)},\bm{w}_8^{(0)}\}$
  are calculated independently, thus these calculation can be
  parallelized.
  \begin{figure}
    \centering
    \begin{tikzpicture}
      \Tree
      [.$\bm{u}$
        [.$\bm{u}_{(5,6,7,8)}$ 
          [.$\bm{u}_{(3,4,5,6,7,8)}$
            [.$\bm{w}_1^{(0)}$ ]
            [.$\bm{w}_2^{(0)}$ ]
          ]
          [.$\bm{u}_{(1,2,5,6,7,8)}$ 
            [.$\bm{w}_3^{(0)}$ ]
            [.$\bm{w}_4^{(0)}$ ]
          ]
        ]
        [.$\bm{u}_{(1,2,3,4)}$
          [.$\bm{u}_{(1,2,3,4,7,8)}$ 
            [.$\bm{w}_5^{(0)}$ ]
            [.$\bm{w}_6^{(0)}$ ]
          ]
          [.$\bm{u}_{(1,2,3,4,5,6)}$ 
            [.$\bm{w}_7^{(0)}$ ]
            [.$\bm{w}_8^{(0)}$ ]
          ]
        ]
      ]
    \end{tikzpicture}
    \caption{Calculating $\bm{w}_j^{(0)}$s in \cref{eq:w-vectors}
      with a binary tree. See Example~\ref{expl:binary-tree} for
      details.}
    \label{fig:pseudo-annih-pol-binary-tree}
  \end{figure}
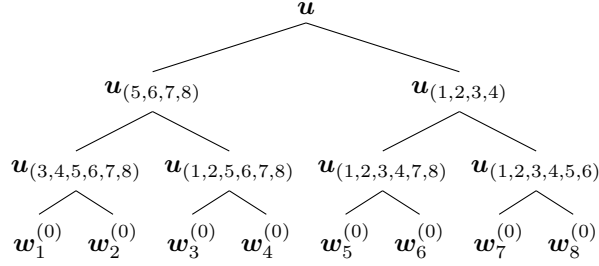
\end{expl}

We define the binary tree used in
Example~\ref{fig:pseudo-annih-pol-binary-tree} as follows.
\begin{defn}[A binary tree used to calculate
  $\bm{w}_1^{(0)},\ldots,\bm{w}_q^{(0)}$ in \cref{eq:w-vectors}]
  \label{defn:binary-tree-2}
  Let $p'$ and $p''$ be integers satisfying $1\le p'\le p''\le q$ and
  $S(p',p'')=\{p',\ldots,p''\}$. For $S=S(1,q)$, 
  define a binary tree $(T_s,V_s)$, where $T_s$ is the set of the
  nodes and $V_s$ is the set of the vertices, as follows:
  \begin{enumerate}
  \item Let $S$ be the root node;
  \item For a node $S(p',p'')$ satisfying $p'<p''$, let it be an
    internal node and append $S(p',p'+\lfloor(p''-p')/2\rfloor)$ and
    $S(p'+\lfloor(p''-p')/2\rfloor+1,p'')$ as its child nodes;
  \item If a node $S(p',p'')$ satisfies $p'=p''$, then let it be a leaf
    node.
  \end{enumerate}
\end{defn}

  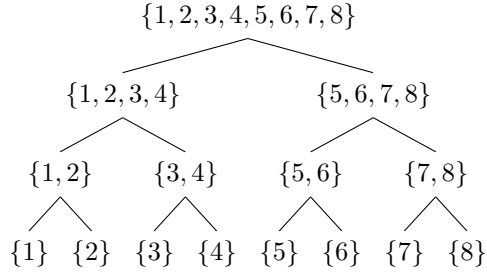
\begin{figure}
    \centering
    \begin{tikzpicture}
      \Tree
      [.$\{1,2,3,4,5,6,7,8\}$
        [.$\{1,2,3,4\}$ 
          [.$\{1,2\}$
            [.$\{1\}$ ]
            [.$\{2\}$ ]
          ]
          [.$\{3,4\}$ 
            [.$\{3\}$ ]
            [.$\{4\}$ ]
          ]
        ]
        [.$\{5,6,7,8\}$
          [.$\{5,6\}$ 
            [.$\{5\}$ ]
            [.$\{6\}$ ]
          ]
          [.$\{7,8\}$ 
            [.$\{7\}$ ]
            [.$\{8\}$ ]
          ]
        ]
      ]
    \end{tikzpicture}
    \caption{The binary tree of Example~\ref{expl:binary-tree}
      (cf.\ Figure~\ref{fig:pseudo-annih-pol-binary-tree}).}
    \label{fig:pseudo-annih-pol-binary-tree-2}
  \end{figure}

\begin{prop}
  \label{prop:binary-splitting}
  Let $S=\{1,\ldots,q\}$ and $I$ be a node in the graph $(T_S,V_S)$ in
  Definition~\ref{defn:binary-tree-2}. For $I$, let
  \begin{equation}
    \label{eq:vi}
    \bm{v}_I=\bm{u}\prod_{j\in S\setminus I}F_j^{m_j},
  \end{equation}
  where $\bm{u}$ and $F_j$ are defined as in \cref{eq:w-vectors}.
  For calculating
  $\bm{w}_1^{(0)},\ldots,\bm{w}_q^{(0)}$, the total number of
  vector-matrix 
  multiplications by matrices $F_p^{m_p}$'s is estimated as
  $O(q\log_2q)$. 
\end{prop}
\begin{proof}
  We show that tracing all the paths of $(T_S,V_S)$ from the root node
  to the leaves enables us to calculate
  $\bm{w}_1^{(0)},\ldots,\bm{w}_q^{(0)}$. First, by
  \cref{eq:vi}, we have $\bm{u}=\bm{u}\prod_{j\in S\setminus
    S}F_j^{m_j}=\bm{v}_S$ and $\bm{w}_p^{(0)}=\bm{u}\prod_{j\in
    S\setminus \{p\}}F_j^{m_j}=\bm{v}_{\{p\}}$ for
  $p=1,\ldots,q$. Now, let $I$ be a parent node and $J$ be its child 
  node. Then we have
  \[
  \bm{v}_J=\bm{v}_I\prod_{j\in I\setminus J}F_j^{m_j},
  \]
  which means the fact that a vector in the child node can be
  calculated by
  using intermediate result in its parent node. For estimating total
  number of multiplications of matrices $F_p^{m_p}$, it requires $q$
  multiplications of $F_p^{m_p}$ in total for calculating vectors in
  all the nodes of depth $k$, and the height of the tree is estimated
  as $O(\log_2 q)$, thus we have the claim.  This completes the
  proof.
\end{proof}



\begin{prop}
  With the binary splitting as in
  Proposition~\ref{prop:binary-splitting}, Algorithm~\ref{alg:upap}
  calculates unit pseudo annihilating polynomials with
  \begin{equation}
    \label{eq:upap:time:binary-splitting}
    O(n^3\max\{1,\log_2 q\})
  \end{equation}
  arithmetic operations in $\K$.
\end{prop}
\begin{proof}
  The number of arithmetic operations in Algorithm~\ref{alg:upap} is
  estimated as follows. First, in line~\ref{alg:upap:ugia}, total
  number of operations for calculating $\bar{\bm{b}}_i$ for
  $i=1,\ldots,q$ is $O(n^3\log_2 q)$.
  Next, in
  line~\ref{alg:upap:vectormatrix}, the number of operations for
  calculating $\bar{\bm{b}}_i$ is bounded above by $O(n^2d_im_i)$ for
  each $i\in\{1,\ldots,q\}$, thus total number of operations for
  $i=1,\ldots,q$ is bounded above by
  $O\left(n^2\sum_{i=1}^q(d_im_i)\right)=O(n^3)$. As a consequence,
  total number of operations becomes as in
  \cref{eq:upap:time:binary-splitting}, which completes the proof.
\end{proof}

\begin{thm}
  \label{thm:binary-splitting}
  With the help of binary splitting technique described in
  Proposition~\ref{prop:binary-splitting}, Algorithm~\ref{alg:umap-p}
  calculates the unit minimal annihilating polynomials with
  \begin{multline}
    O\biggl( n^3\max\{1,\log_2 q\}
      +n^2\deg(\pi_A(\lambda))
      +n^2\sum_{j=1}^n\deg\pi'_{A,j}(\lambda)
      \\
      +n^2\sum_{j=1}^{n}\sum_{k=1}^{q''_j}
      d_k\{(m_k-\rho_{kj})+k\delta_{k,j}\} \biggr)
    \label{eq:alg:umap-p-binary-splitting-time}
  \end{multline}
  arithmetic
  operations in $\K$.
\end{thm}
\begin{proof}
  In Proposition~\ref{prop:alg:umap-p-time}, now the first term in
  \cref{eq:umap-p:time:upap} is replaced with
  \cref{eq:upap:time:binary-splitting}, thus
  \cref{eq:alg:umap-p-time} is replaced with
  \cref{eq:alg:umap-p-binary-splitting-time},
  which proves the theorem.
\end{proof}

\begin{rem}
  \label{rem:thm:binary-splitting:minpoly}
  If we have the irreducible factorization of the minimal polynomial
  $\pi_A(\lambda)$ as in \cref{eq:minpol},
  \cref{eq:alg:umap-p-binary-splitting-time} becomes
  \begin{multline*}
    O\biggl(n^2(\deg(\pi_A(\lambda)))\max\{1,\log_2 q\}
      +n^2\deg(\pi_A(\lambda))
      +n^2\sum_{j=1}^n\deg\pi'_{A,j}(\lambda)
      \\
      +n^2\sum_{j=1}^{n}\sum_{k=1}^{q''_j}
      d_k\{(l_k-\rho_{kj})+k\delta_{k,j}\} \biggr).
  \end{multline*}
\end{rem}

\begin{rem}
  \label{rem:thm:binary-splitting:parallel}
  As well as in Remark~\ref{rem:prop:umap:parallel}, if we distribute
  each processes to $M$ processors, the computing time in
  \cref{eq:alg:umap-p-binary-splitting-time}
  is
  \begin{multline*}
    O
    \biggl(
    \frac{n^3}{\min\{M,q\}}\max\{1,\log_2q\}
    +\frac{n^2}{\min\{M,q\}}\deg(\pi_A(\lambda))
    +\frac{n^2}{M}\sum_{j=1}^n\deg\pi'_{A,j}(\lambda)
    \\
    +\frac{n^2}{M}\sum_{j=1}^n\sum_{k=1}^{q''_j}
    d_k\{(m_k-\rho_{kj})+k\delta_{k,j}\}
    \biggr).
  \end{multline*}
\end{rem}


\section{Concluding remarks}
\label{sec:concl}

We have described a new method for calculating, for a given matrix
over a field of characteristic zero, all the unit minimal annihilating
polynomials. A key of the proposed method is the use of pseudo
annihilating polynomials. The proposed method can be regarded
therefore a kind of extension or variant of Wiedemann's
method. Note that the algorithm for calculating pseudo annihilating
polynomials is designed by exploiting the irreducible factorization of
the characteristic polynomial.
Accordingly, the possibility that pseudo annihilating
polynomials used in the resulting algorithm are true minimal
annihilating polynomials is expected to be high. This is a significant
feature of the proposed method. Another feature is the
parallelizability: parallel processing can be applied to main blocks
of the resulting algorithm.

Since the set of all minimal annihilating polynomials of unit vectors
has a lot of information on a matrix, the algorithm presented in the
present paper has several applications in exact linear algebra,
especially in eigenproblems. Let us emphasis here the fact that pseudo
annihilating polynomials are very useful to design new efficient
algorithms on exact linear algebra. Some relevant applications will be
discussed in separate papers.

\bibliographystyle{elsarticle-harv}
\bibliography{terui-e}

\end{document}